\documentclass{amsart}

\usepackage{Rdim}

\title[The Rouquier dimension for a regular ring]{The Rouquier dimension of the category of perfect complexes over a regular ring}
\date{\today}



\author[J.C.~Letz]{Janina~C.~Letz}
\address{Janina~C.~Letz,
Faculty of Mathematics,
Bielefeld University,
PO Box 100 131,
33501 Bielefeld,
Germany}
\email{jletz@math.uni-bielefeld.de}

\begin{document}

\begin{abstract}
We show that the Rouquier dimension of the category of perfect complexes over a regular ring is precisely the Krull dimension of the ring. 
Previously, it was known that the Krull dimension is an upper bound, the lower bound however was not known in general.
In particular, for regular local rings this result is new. 
More generally, we show that a lower bound of the Rouquier dimension is given by the maximal length of a regular sequence.
\end{abstract}

\maketitle


\section*{Introduction}

The Rouquier dimension measures the complexity of a triangulated category. It is the infimum, taken over all strong generators, of the generation time. Determining the Rouquier dimension for a triangulated category is generally rather difficult. It has been determined for some bounded derived categories; see for example \cite{Rouquier:2006,Rouquier:2008,Chen/Ye/Zhang:2008,Elagin:2022}. However, often only upper or lower bounds are known; see for example \cite{Krause/Kussin:2006,Rouquier:2008,Oppermann:2009,Aihara/Takahashi:2015,Huang/Zheng:2020}. For the category of perfect complexes over a regular ring we are able to determine the Rouquier dimension explicitly.

\begin{introthm} 
Let $A$ be a commutative noetherian regular ring. Then 
\begin{equation*}
\Rdim(\Perf{A}) = \dim(A)\,.
\end{equation*}
\end{introthm}

This result was previously known when $A$ is a regular ring of finite type over a field; see \cite[Theorem~7.17]{Rouquier:2008}. However, for general regular rings, it was only known that the Krull dimension $\dim(A)$ is an upper bound; see \cite{Krause/Kussin:2006}. We show that the Krull dimension $\dim(A)$ is a lower bound for any regular ring $A$. 

More generally, we consider triangulated categories with an action by a graded-commutative ring. Then depth, that is the maximal length of a regular sequence, provides a lower bound for Rouquier dimension.

\begin{introthm}
Let $R$ be a graded-commutative ring and $\cat{T}$ an Ext-noetherian $R$-linear triangulated category. Let $M \in \cat{T}$ and $\fa \subseteq R$ an ideal with $\fa \Hom[*]{\cat{T}}{M}{M} \neq \Hom[*]{\cat{T}}{M}{M}$. Then
\begin{equation*}
\depth_R(\fa,\Hom[*]{\cat{T}}{M}{M}) \leq \Rdim(\cat{T})\,.
\end{equation*}
\end{introthm}

This result is contained in \cref{depth_leq_Rdim}. Similar as to \cite{Bergh/Iyengar/Krause/Oppermann:2010,Bai/Cote:2023}, we construct a non-zero composition of ghost maps from a regular sequence. Unlike \cite{Bergh/Iyengar/Krause/Oppermann:2010,Bai/Cote:2023}, we construct ghost maps from arbitrary elements; see \cref{kos_map_G_ghost}. The regularity of the sequence ensures that the composition of ghost maps is non-zero. 



\begin{ack}
This work is inspired by joint work with Antonia Kekkou and Marc Stephan. 
I thank Srikanth Iyengar for helpful discussions

The author was partly funded by the Deutsche Forschungsgemeinschaft (DFG, German Research Foundation)---Project-ID 491392403---TRR 358. 
\end{ack}

\section{Rouquier dimension}

Let $\cat{T}$ be a triangulated category. Subcategories are assumed to be strict and full. 

\subsection{Level}

Let $\cat{T}$ be a triangulated category and $G \in \cat{T}$. We denote by $\thick_\cat{T}(G)$ the smallest thick subcategory of $\cat{T}$ that contains $G$.
By \cite[Section~2.2]{Bondal/VanDenBergh:2003}, such a thick subcategory is equipped with a filtration given by
\begin{equation*}
\thick^n(G) = \smd(\add(G)^{\star n}) = \smd(\add(G) \star \cdots \star \add(G))
\end{equation*}
for any $n \geq 0$. Here, $\smd(-)$ denotes the closure under direct summands and $\add(-)$ denotes the closure under finite direct sums, suspensions and desuspensions. Further, for subcategories $\cat{A}$ and $\cat{B}$, the subcategory $\cat{A} \star \cat{B}$ consists of all objects $M$ such that there is an exact triangle $A \to M \to B \to \susp A$ with $A \in \cat{A}$ and $B \in \cat{B}$. 

Following \cite[2.3]{Avramov/Buchweitz/Iyengar/Miller:2010}, for any $M \in \cat{T}$, the \emph{$G$-level of $M$} is
\begin{equation*}
\level_\cat{T}^G(M) \coloneqq \inf\set{n | M \in \thick_\cat{T}^n(G)}.
\end{equation*}
One can think of $\level_\cat{T}^G(M)$ as the number of extensions necessary to build $M$ from $G$. 
One has
\begin{equation*}
\level_\cat{T}^G(M) < \infty \iff M \in \thick_\cat{T}(G)\,.
\end{equation*}

\subsection{Generation time and Rouquier dimension}

When $\thick_\cat{T}(G) = \cat{T}$, the object $G$ is a \emph{generator} of $\cat{T}$. The \emph{generation time} of a generator $G$ is
\begin{equation*}
\gentime_\cat{T}(G) \coloneqq \sup \set{\level_\cat{T}^G(X)-1 | X \in \cat{T}}.
\end{equation*}
If $\gentime_\cat{T}(G) < \infty$, meaning there is a universal upper bound for $G$-level, then $G$ is a \emph{strong generator}. 

When the triangulated category $\cat{T}$ has a strong generator, every generator is strong. In this case, one is interested in minimal generation time possible. The \emph{Rouquier dimension} of a triangulated category $\cat{T}$ is defined as
\begin{equation*}
\Rdim(\cat{T}) \coloneqq \inf\set{\gentime_\cat{T}(G) | G \text{ a generator of } \cat{T}}.
\end{equation*}
This invariant was introduced in \cite[Definition~3.2]{Rouquier:2008}. 

\subsection{Ghost maps} \label{sec:ghost}

A morphism $f \colon M \to N$ is \emph{$G$-ghost}, if the composition $\susp^n G \to M \xrightarrow{f} N$ is zero for any morphism $\susp^n G \to M$ and any $n \in \BZ$. Non-zero ghost maps provide an obstruction to generation: If there exists a non-zero composition 
\begin{equation*}
M = M_0 \to M_1 \to M_2 \to \cdots \to M_{t-1} \to M_t
\end{equation*}
where each morphism is $G$-ghost, then $\level_\cat{T}^G(M) > t$. This is the ghost lemma; see for example \cite[Lemma~2.2]{Beligiannis:2008}. 

\section{Depth and Rouquier dimension}

Let $R$ be a graded-commutative ring.
We think of $R$ as a disjoint union of its graded components, meaning every element is homogeneous and every ideal is generated by homogeneous elements.
For $x \in R$, we denote by $|x|$ the degree of $x$.

A triangulated category $\cat{T}$ is \emph{$R$-linear}, if the graded Hom-sets
\begin{equation*}
\Hom[*]{\cat{T}}{M}{N} \coloneqq \coprod_{n \in \BZ} \Hom{\cat{T}}{M}{\susp^n N}
\end{equation*}
are graded $R$-modules and the composition is $R$-bilinear. If $\Hom[*]{\cat{T}}{M}{N}$ is noetherian as an $R$-module for all objects $M,N \in \cat{T}$, then we say the $R$-linear triangulated category $\cat{T}$ is \emph{Ext-noetherian}. 

\subsection{Koszul objects}

Let $R$ be a graded-commutative ring and $\cat{T}$ an $R$-linear triangulated category.
Any $x \in R$ provides a natural morphism
\begin{equation*}
x = x(M) = x \cdot \id_M \colon M \to \susp^{|x|} M\,.
\end{equation*}

The \emph{Koszul object} of a sequence $x_1, \ldots, x_t \in R$ on an object $M$ is
\begin{equation*}
\kosobj{M}{(x_1, \ldots, x_t)} \coloneqq \begin{cases}
M & t = 0\,, \\
\cone(x(M)) & t = 1\,, \\
\kosobj{(\kosobj{M}{(x_1, \ldots, x_t)})}{x_t} & t > 1\,.
\end{cases}
\end{equation*}
That is, the Koszul object $\kosobj{M}{x}$ with respect to an element $x \in R$ is constructed through the exact triangle
\begin{equation*}
\susp^{-1} \kosobj{M}{x} \xrightarrow{\epsilon(x)} M \xrightarrow{x} \susp^{|x|} M \to \kosobj{M}{x}\,.
\end{equation*}
We call this the \emph{defining exact triangle of $\kosobj{M}{x}$}. The morphism $\epsilon(x)$ plays a central role in the following. It need not be natural in $M$.

\subsection{Regular sequences}

By \cite{Kekkou/Letz/Stephan:2025}, for $M \in \cat{T}$ and a sequence $x_1, \ldots, x_t \in R$, the following are equivalent
\begin{enumerate}
\item \label{regular_equiv:end} $x_1, \ldots, x_t$ is a $\Hom[*]{\cat{T}}{M}{M}$-regular sequence;
\item Each morphism in the composition
\begin{equation*}
\susp^{-t} \kosobj{M}{(x_1, \ldots, x_t)} \xrightarrow{\epsilon(x_t)} \susp^{-t+1} \kosobj{M}{(x_1, \ldots, x_{t-1})} \to \cdots \to \susp^{-1} \kosobj{M}{x_1} \xrightarrow{\epsilon(x_1)} M
\end{equation*}
is $M$-ghost and the composition is non-zero.
\end{enumerate}
If these conditions hold, we say $x_1, \ldots, x_t$ is a \emph{$M$-regular sequence}. In this paper we only use condition \cref{regular_equiv:end}. 

For an ideal $\fa \subseteq R$ the \emph{$\fa$-depth} $\depth_R(\fa,\cM)$ of a graded $R$-module $\cM$ is the supremum over the length of $\cM$-regular sequences in $\fa$ when $\fa \cM \neq \cM$ and $\infty$ otherwise. 



For a triangulated category, which has a generator, it is always possible to construct a generator such that its endomorphism ring has depth zero. 
However, for any object $M \in \cat{T}$, an element $x \in R$ induces a $G$-ghost map with source $\kosobj{M}{x^n}$ for some $n \gg 0$. 

\begin{lemma} \label{kos_map_G_ghost}
Let $R$ be a graded-commutative ring and $\cat{T}$ an $R$-linear triangulated category.
We fix $G,M \in \cat{T}$ such that $\Hom[*]{\cat{T}}{G}{M}$ a noetherian $R$-module.
For $x \in R$, the composition
\begin{equation*}
\susp^{-1} \kosobj{M}{x^{n+1}} \xrightarrow{\epsilon(x^{n+1})} M \xrightarrow{x^n} \susp^{n|x|} M
\end{equation*}
is $G$-ghost for $n \gg 0$.
\end{lemma}
\begin{proof}
For convenience, we write $\ch{G}{-} \coloneqq \Hom[*]{\cat{T}}{G}{-}$.
As $\ch{G}{M}$ is noetherian, its submodule of $(x)$-torsion elements
\begin{equation*}
\Gamma_{(x)} \ch{G}{M} = \set{f \in \ch{G}{M} | x^n f = 0 \text{ for some } n \geq 0},
\end{equation*}
is finitely generated.
Hence, there exists $n \geq 0$ such that $x^n \Gamma_{(x)} \ch{G}{M} = 0$.
In particular, this means
\begin{equation*}
\Gamma_{(x)} \ch{G}{M} = \ker(\ch{G}{M} \xrightarrow{x^n} \ch{G}{M}) = \ker(\ch{G}{M} \xrightarrow{x^{n+1}} \ch{G}{M})\,.
\end{equation*}

Applying $\ch{G}{-}$ to the defining exact triangle of $\kosobj{M}{x^{n+1}}$ we obtain the exact sequence
\begin{equation*}
\ch{G}{\susp^{-1} \kosobj{M}{x^{n+1}}} \xrightarrow{\ch{G}{\epsilon(x^{n+1})}} \ch{G}{M} \xrightarrow{x^{n+1}} \ch{G}{\susp^{(n+1)|x|} M}\,.
\end{equation*}
Hence, the first morphism factors through $\Gamma_{(x)} \ch{G}{M}$ and the composition
\begin{equation*}
\ch{G}{\susp^{-1} \kosobj{M}{x^{n+1}}} \xrightarrow{\ch{G}{\epsilon(x^{n+1})}} \ch{G}{M} \xrightarrow{x^n} \ch{G}{\susp^{n|x|} M}
\end{equation*}
is zero and the claim holds.
\end{proof}

From this \namecref{kos_map_G_ghost}, we can construct from a sequence of elements in $R$ a composition of the same length of $G$-ghost maps. 
In general, this composition need not be non-zero. 
To resolve this we consider regular sequences. 

\begin{lemma} \label{kos_map_regular_nonzero}
Let $R$ be a graded-commutative ring and $\cat{T}$ an $R$-linear triangulated category.
If $x_1, \ldots, x_t$ is an $M$-regular sequence, then the composition
\begin{equation} \label{kos_map_regular_nonzero:comp}
\susp^{-t} \kosobj{M}{(x_1^{n_1+1}, \ldots, x_t^{n_t+1})} \xrightarrow{\epsilon(x_1^{n_1+1}) \cdots \epsilon(x_t^{n_t+1})} M \xrightarrow{x_1^{n_1} \cdots x_t^{n_t}} \susp^{n_1 |x_1| + \cdots + n_t |x_t|} M\,.
\end{equation}
is non-zero for any $n_1, \ldots, n_t \geq 0$.
\end{lemma}
\begin{proof}
For convenience, we write $\hh{M}{-} \coloneqq \Hom[*]{\cat{T}}{-}{M}$.
We set $\cM \coloneqq \hh{M}{M}$. 
This is a graded $R$-module and we denote its twist by
\begin{equation*}
\cM(n) \coloneqq \Hom[*]{\cat{T}}{M}{\susp^{-n} M} = \hh{M}{\susp^{n} M}\,.
\end{equation*}
As $x_1, \ldots, x_t$ is an $M$-regular sequence, it is a $\cM$-regular sequence and $\id_M \notin \langle x_1, \ldots, x_t \rangle \cM$. 

We use induction on $t$ to show that there is a natural isomorphism
\begin{equation*}
\hh{M}{\susp^{-t} \kosobj{M}{(x_1, \ldots, x_t)}} \xrightarrow{\cong} \cM/\langle x_1, \ldots, x_t \rangle \cM
\end{equation*}
that identifies the canonical morphism $\susp^{-t} \kosobj{M}{(x_1, \ldots, x_t)} \to M$ with $[\id_M]$. 

For $t = 0$, there is nothing to show.
For $t > 0$, we apply $\hh{M}{-}$ to the defining exact triangle of $\susp^{-t} \kosobj{M}{(x_1, \ldots, x_t)}$.
This yields a commutative diagram
\begin{equation*}
{\begin{tikzcd}[column sep=small]
0 \ar[r] & \hh{M}{\susp^{|x_t|-t+1} \kosobj{M}{\bmx_{t-1}}} \ar[r,"x_t"] \ar[d,"\cong"] & \hh{M}{\susp^{-t+1} \kosobj{M}{\bmx_{t-1}}} \ar[r,"{\epsilon(x_t)}"] \ar[d,"\cong"] & \hh{M}{\susp^{-t} \kosobj{M}{\bmx_t}} \ar[d,dashed,"\cong"] \ar[r] & 0 \\
0 \ar[r] & (\cM/\langle \bmx_{t-1} \rangle \cM)(|x_t|) \ar[r,"x_t"] & \cM/\langle \bmx_{t-1} \rangle \cM \ar[r] & \cM/\langle \bmx_t \rangle \cM \ar[r] & 0 \nospacepunct{,}
\end{tikzcd}}
\end{equation*}
where $\bmx_t = x_1, \ldots, x_t$ and $\bmx_{t-1} = x_1, \ldots, x_{t-1}$. 
The rows are short exact sequences as $x_t$ is regular on the given modules.
The first two vertical maps exist and are isomorphisms by the induction hypothesis.
Hence, the dashed vertical map exists and is an isomorphism. 
By the induction hypothesis and the commutativity, the canonical morphism $\susp^{-t} \kosobj{M}{(x_1, \ldots, x_t)} \to M$ is identified with $[\id_M]$. 

By \cite[Theorem~16.1]{Matsumura:1989}, the sequence $x_1^{n_1}, \ldots, x_t^{n_t}$ is an $\cM$-regular sequence as well. So we obtain a commutative diagram
\begin{equation*}
\begin{tikzcd}[column sep=large]
\Hom[*]{\cat{T}}{\susp^{-t} \kosobj{M}{\bmx^{\bmn+\bm1}}}{M} \ar[d,"\cong"] \ar[r,"{x_1^{n_1} \cdots x_t^{n_t}}"] & \Hom[*]{\cat{T}}{\susp^{-t} \kosobj{M}{\bmx^{\bmn+\bm1}}}{\susp^{n_1|x_1|+\cdots+n_t|x_t|} M} \ar[d,"\cong"] \\
\cM/\langle \bmx^{\bmn+\bm1} \rangle \cM \ar[r,"{x_1^{n_1} \cdots x_t^{n_t}}"] & (\cM/\langle \bmx^{\bmn+\bm1} \rangle \cM)(n_1 |x_1| + \cdots + n_t |x_t|) \nospacepunct{,}
\end{tikzcd}
\end{equation*}
where $\bmx^{\bmn} = x_1^{n_1}, \ldots, x_t^{n_t}$ and $\bmx^{\bmn+\bm1} = x_1^{n_1+1}, \ldots, x_t^{n_t+1}$. 
The canonical morphism $\susp^{-t} \kosobj{M}{\bmx^{\bmn+\bm1}} \to M$ in the top left corner is mapped to the composition \cref{kos_map_regular_nonzero:comp} in the top right corner. 
Along the isomorphism it identifies with $[x_1^{n_1} \cdots x_t^{n_t}]$ in the bottom right corner.
Since $x_1, \ldots, x_t$ is $\cM$-regular, we have $[x_1^{n_1} \cdots x_t^{n_t}] \neq 0$ in $\cM/\langle x_1^{n_1+1}, \ldots, x_t^{n_t+1} \rangle \cM$. 
This yields the claim.
\end{proof}

Combining \cref{kos_map_G_ghost,kos_map_regular_nonzero} we obtain a lower bound for generation time.

\begin{proposition} \label{depth_leq_gentime}
Let $R$ be a graded-commutative ring and $\cat{T}$ an $R$-linear triangulated category.
Let $G,M \in \cat{T}$ such that $\Hom[*]{\cat{T}}{G}{M}$ is noetherian as an $R$-module and $\fa \subseteq R$ an ideal with $\fa \Hom[*]{\cat{T}}{M}{M} \neq \Hom[*]{\cat{T}}{M}{M}$.
Then
\begin{equation*}
\depth_R(\fa,\Hom[*]{\cat{T}}{M}{M}) \leq \gentime_\cat{T}(G)\,.
\end{equation*}
\end{proposition}
\begin{proof}
Let $\bmx = x_1, \ldots, x_t$ be an $M$-regular sequence in $\fa$.
By \cref{kos_map_G_ghost}, there exists $\bmn = (n_1, \ldots, n_t) \in \BN_0^t$ such that
\begin{equation*}
\susp^{-s} \kosobj{M}{\bmx_s^{\bmn_s+\bm1}} \xrightarrow{\epsilon(x_s^{n_s+1})} \susp^{-s+1} \kosobj{M}{\bmx_{s-1}^{\bmn_{s-1}+\bm1}} \xrightarrow{x_s^{n_s}} \susp^{n_s|x_s|-s+1} \kosobj{M}{\bmx_{s-1}^{\bmn_{s-1}+\bm1}}
\end{equation*}
is $G$-ghost for any $1 \leq s \leq t$. Here $\bmx_s^{\bmn_s+\bm1}$ stands for the sequence $x_1^{n_1+1}, \ldots, x_s^{n_s+1}$.

As $x_1, \ldots, x_t$ act naturally on $\cat{T}$, multiplication by $x_r^{n_r}$ commutes with the map $\epsilon(x_s^{n_s+1})$ for any $1 \leq r,s \leq t$. Hence the composition
\begin{equation*}
\susp^{-t} \kosobj{M}{(x_1^{n_1+1}, \ldots, x_t^{n_t+1})} \xrightarrow{\epsilon(x_t^{n_t+1}) \cdots \epsilon(x_1^{n_1+1})} M \xrightarrow{x_1^{n_1} \cdots x_t^{n_t}} \susp^{n_1 |x_1| + \cdots + n_t |x_t|} M
\end{equation*}
is the composition of the $t$ morphisms above, each of which is $G$-ghost.
As $x_1, \ldots, x_t$ is $M$-regular, this composition is non-zero by \cref{kos_map_regular_nonzero}. 

Now the ghost lemma yields
\begin{equation*}
t \leq \level_\cat{T}^G(M)-1 \leq \gentime_\cat{T}(G)\,;
\end{equation*}
see \cref{sec:ghost}. 
As $\fa \Hom[*]{\cat{T}}{M}{M} \neq \Hom[*]{\cat{T}}{M}{M}$ we obtain
\begin{equation*}
\depth_R(\fa,\Hom[*]{\cat{T}}{M}{M}) \leq \gentime_\cat{T}(G)\,. \qedhere
\end{equation*}
\end{proof}

As the \namecref{depth_leq_gentime} holds for any object $G$, we obtain a lower bound for Rouquier dimension. 

\begin{theorem} \label{depth_leq_Rdim}
Let $R$ be a graded-commutative ring and $\cat{T}$ an Ext-noetherian $R$-linear triangulated category.
Let $M \in \cat{T}$ and $\fa \subseteq R$ an ideal with $\fa \Hom[*]{\cat{T}}{M}{M} \neq \Hom[*]{\cat{T}}{M}{M}$.
Then
\begin{equation*}
\depth_R(\fa,\Hom[*]{\cat{T}}{M}{M}) \leq \Rdim \cat{T}\,. \pushQED {\qed} \qedhere \popQED
\end{equation*}
\end{theorem}

\section{Regular rings}

Let $A$ be a commutative noetherian regular ring.
Then $A$, viewed as a graded ring concentrated in degree zero, acts on the bounded derived category of finitely generated $A$-modules $\dbcat{\mod{A}}$.
As $A$ is regular, the bounded derived category $\dbcat{\mod{A}}$ of finitely generated modules coincides with the category of perfect $A$-complexes. 
This category is Ext-noetherian as an $A$-linear triangulated category. 
The ring $A$, viewed as a complex concentrated in degree 0, is a generator of $\dbcat{\mod{A}}$. 
By \cite[Proposition~7.15]{Rouquier:2008} and \cite[Proposition~2.6]{Krause/Kussin:2006}, the generation time of $A$ is
\begin{equation} \label{Rdim_leq_dim}
\gentime_{\dbcat{\mod{A}}}(A) = \dim(A)\,.
\end{equation}
Combining this with \cref{depth_leq_Rdim} we obtain our main result.

\begin{theorem} \label{Rdim_regular_ring}
Let $A$ be a commutative noetherian regular ring. Then 
\begin{equation*}
\Rdim(\dbcat{\mod{A}}) = \dim(A)\,.
\end{equation*}
\end{theorem}
\begin{proof}
First, we show the claim for a regular local ring $A$ with maximal ideal $\fm$. We consider $A$ as a complex concentrated in degree zero. As $A = \Hom[*]{\dbcat{\mod{A}}}{A}{A}$, \cref{depth_leq_Rdim} yields
\begin{equation*}
\begin{aligned}
\dim(A) = \depth_A(\fm,A) &\leq \Rdim(\dbcat{\mod{A}}) \\
&\leq \gentime_{\dbcat{\mod{A}}}(A) = \dim(A)\,.
\end{aligned}
\end{equation*}
Hence, equality holds.

Now, let $R$ be a regular ring, not necessarily local. 
When $\dim(A) = \infty$, then $A$ is a generator of $\dbcat{\mod{A}}$, through not a strong generator. Hence $\dbcat{\mod{A}}$ has no strong generator and $\Rdim(\dbcat{\mod{A}}) = \infty$. 

We now assume $\dim(A) < \infty$.
Then there exists a maximal ideal $\fm$ of $A$ such that $\dim(A) = \dim(A_\fm)$. 
For any strong generator $G$ of $\dbcat{\mod{A}}$ one has
\begin{equation*}
\dim(A) = \dim(A_\fm) \leq \gentime_{\dbcat{\mod{A_\fm}}}(G_\fm) \leq \gentime_{\dbcat{\mod{A}}}(G)\,;
\end{equation*}
for the second inequality see \cite[Section~3]{Letz:2021}. 
Hence, $\dim(A)$ provides a lower bound for the Rouquier dimension of $\dbcat{\mod{A}}$. Combined with \cref{Rdim_leq_dim} this yields the claim.
\end{proof}

\bibliographystyle{amsalpha}
\bibliography{refs}

\end{document}